\documentclass[11pt]{article}
\usepackage{tikz}
\usepackage{ucs}
\usepackage{amssymb}
\usepackage{amsthm}
\usepackage{amsmath}
\usepackage{latexsym}
\usepackage[cp1251]{inputenc}
\usepackage[english]{babel}
\usepackage{graphicx, color}
\usepackage{wrapfig}
\usepackage{caption}
\usepackage{subcaption}
\usepackage{txfonts}
\usepackage{lmodern}
\usepackage{mathrsfs}
\usepackage{algorithm}
\usepackage{multicol}
\usepackage{enumerate}
\usepackage{titlesec}
\usepackage{paralist}
\usepackage{mathtools}
\usepackage[hidelinks]{hyperref}

\newtheorem{thm}{Theorem}
\newtheorem{lemma}[thm]{Lemma}
\newtheorem{claim}{Claim}[thm]
\newtheorem{conj}[thm]{Conjecture}

\newtheorem{defn}[thm]{Definition}
{}
\newtheorem{theorem}[thm]{Theorem}

\newtheorem{example}[thm]{Example}
\newtheorem*{example*}{Example}

\newtheorem{remark}[thm]{Remark}
\newtheorem{observation}[thm]{Observation}
\newtheorem*{definition*}{Definition}
\newtheorem*{remark*}{Remark}

\titleformat{\paragraph}
{\normalfont\normalsize\bfseries}{\theparagraph}{1em}{}
\titlespacing*{\paragraph}
{0pt}{3.25ex plus 1ex minus .2ex}{1.5ex plus .2ex}

\newcommand*{\myproofname}{Proof}

\usepackage{wasysym}

\usepackage[top=2cm, bottom=2cm, left=2cm, right=2cm]{geometry}

\def\qed{\hfill\ifhmode\unskip\nobreak\fi\qquad\ifmmode\Box\else\hfill$\Box$\fi}

\title{On the $(1^2,2^4)$-packing edge-coloring of subcubic graphs}
\date{\today}

\author{
Xujun Liu\thanks{Department of Foundational Mathematics, Xi'an Jiaotong-Liverpool University, Suzhou, Jiangsu Province, 215123, China, xujun.liu@xjtlu.edu.cn; the research of X. Liu was supported by the Natural Science Foundation of the Jiangsu Higher Education Institutions of China (Grant No. 22KJB110025) and the Research Development Fund RDF-21-02-066 of Xi'an Jiaotong-Liverpool University.} \and
Gexin Yu\thanks{Department of Mathematics, William \& Mary, Williamsburg, VA, USA, gyu@wm.edu.}
 }

\begin{document}
	\maketitle

\begin{abstract}
An induced matching in a graph $G$ is a matching such that its end vertices also induce a matching.  A $(1^{\ell}, 2^k)$-packing edge-coloring of a graph $G$ is a partition of its edge set into disjoint unions of $\ell$ matchings and $k$ induced matchings. Gastineau and Togni (2019), as well as Hocquard, Lajou, and Lu\v zar (2022), have conjectured that every subcubic graph is $(1^2,2^4)$-packing edge-colorable. In this paper, we confirm that their conjecture  is true (for connected subcubic graphs with more than $70$ vertices). Our result is sharp due to the existence of subcubic graphs that are not $(1^2,2^3)$-packing edge-colorable.
\end{abstract}

\section{Introduction}

A proper $k$-edge-coloring of a graph $G$ is a partition of its edges into disjoint unions of $k$ matchings. The chromatic index of a graph $G$ is the minimum $k$ such that $G$ has a proper $k$-edge-coloring. By the famous theorem of Vizing~\cite{V1}, we know the chromatic index of a simple graph $G$ with maximum degree $\Delta$ is either $\Delta$ or $\Delta+1$, where the former class of graphs is called {\em Class I} and the latter is called {\em Class II}. A strong edge-coloring of a graph is an assignment of colors to the edge set such that every color class induces a matching. The strong chromatic index of a graph is the minimum number of colors needed to complete a strong edge-coloring. The notion of strong edge-coloring was first introduced by Fouquet and Jolivet~\cite{FJ1}, and later it was conjectured by Erd\H{o}s and Ne\v set\v ril~\cite{EN1} that the strong chromatic index of a graph $G$ is bounded by $\frac{5}{4} \Delta^2$ if $\Delta$ is even and $\frac{5}{4} \Delta^2 - \frac{1}{2} \Delta + \frac{1}{4}$ if $\Delta$ is odd. Many researchers have made progress toward this conjecture and it includes but not limited to~\cite{A1,CDYZ1, CKKR1, FKS1, LMSS1, MR1}.

For a non-decreasing sequence $(s_1, \ldots, s_k)$ of positive integers, an $S$-packing edge-coloring of a graph $G$ is a partition of its edge set $E(G)$ into $E_1, \ldots, E_k$ such that the distance between every pair of distinct edges $e_1,e_2 \in E_i$ is at least $s_i+1$, where the distance between edges is defined to be their corresponding vertex distance in the line graph of $G$. To simplify the notation, we denote repetitions of the same numbers in $S$ using exponents and omit the word ``packing'' in this paper. For example, $(1,1,2,2,2,2)$-packing edge-coloring is denoted by $(1^2,2^4)$-edge-coloring. 

By the definition of $S$-packing edge-coloring, a $(1^{\ell})$-edge-coloring is a proper edge coloring using $\ell$ colors and a $(2^k)$-edge-coloring is a strong edge-coloring using $k$ colors. The notion of $S$-packing edge-coloring was first generalized from its vertex counterpart (See~\cite{BKL1,BF1,BKRW1,FKL1,GT2}) by Gastineau and Togni~\cite{GT1}, who studied the $S$-packing edge-coloring of subcubic graphs with a prescribed number of $1$'s in the sequence. By Vizing's famous result~\cite{V1}, a cubic graph either has a $(1^3)$-edge-coloring or a $(1^4)$-edge-coloring. For strong edge-coloring, Andersen~\cite{A1} and independently Horak, Qing, and Trotter\cite{HQT1} proved that every subcubic graph has a $(2^{10})$-edge-coloring, confirming the conjecture of Erd\H{o}s and Ne\v set\v ril for $\Delta = 3$. For graphs $G$ with $\Delta(G) = 4$, Huang, Santana, and Yu~\cite{HSY1} proved that $G$ is $(2^{21})$-edge-colorable and it is only one color away from the conjectured bound of Erd\H{o}s and Ne\v set\v ril. For subcubic planar graphs, Kostochka, Li, Ruksasakchai, Santana, Wang, and Yu~\cite{KLRSWY1} showed that it is $(2^9)$-edge-colorable and it is sharp due to the existence of non-$(2^8)$-edge-colorable subcubic graphs. When $\Delta$ is large, the current best upper bound, $1.772 \Delta^2$, was proved by Hurley, de Joannis de Verclos, and Kang~\cite{HJK1} using the probabilistic method.

Payan~\cite{P1} proved that every subcubic graph admits a $(1^3,2)$-edge-coloring and it is also sharp as the Petersen graph is not $(1^3,3)$-edge-colorable. An immediate corollary from the result ``every subcubic graph is $(2^{10})$-edge-colorable'' is that every subcubic graph is $(1, 2^9)$-edge-colorable. Gastineau and Togni~\cite{GT1} found that there are subcubic graphs that are not $(1,2^6)$-edge-colorable and asked an open question ``Is it true that all cubic graphs are $(1,2^7)$-edge-colorable?''. They also conjectured that every subcubic graph is $(1^2,2^4)$-edge-colorable. Hocquard, Lajou, and Lu\v zar~\cite{HLL1,HLL2} observed that a $(1^{\ell}, 2^k)$-packing edge-coloring can be viewed as an intermediate coloring between proper edge coloring and strong edge-coloring. They showed that every subcubic graph is $(1,2^8)$-edge-colorable and $(1^2,2^5)$-edge-colorable. Furthermore, they conjectured that every subcubic graphs is $(1,2^7)$-edge-colorable and $(1^2,2^4)$-edge-colorable.

\begin{conj}[Hocquard, Lajou, and Lu\v zar~\cite{HLL2}]\label{conj1}
Every subcubic graph is $(1,2^7)$-edge-colorable.    
\end{conj}

\begin{conj}[Gastineau and Togni~\cite{GT1}, Hocquard, Lajou, and Lu\v zar~\cite{HLL2}]\label{conj2}
Every subcubic graph is $(1^2,2^4)$-edge-colorable.    \end{conj}

Very recently, Liu, Santana, and Short~\cite{LSS1} confirmed that Conjecture~\ref{conj1} is true and extended it to subcubic multigraphs. Their result is also sharp since there exist non-$(1,2^6)$-edge-colorable subcubic graphs.

In this paper, we confirm Conjecture~\ref{conj2} for connected subcubic graphs with more than $70$ vertices. Our result is also sharp since there is a subcubic graph that is not $(1^2,2^3)$-edge-colorable.

\begin{theorem}
Let $G$ be a connected subcubic graph with more than $70$ vertices. Then $G$ has a $(1^2,2^4)$-edge-coloring. 
\end{theorem}

\begin{example}
There exist graphs that are not $(1^2,2^3)$-edge-colorable. Here is one example. 
 Let $G$ be the graph obtained by subdividing one edge of $K_{3,3}$ exactly once (See Figure~\ref{ex-1}). 
Note that $G$ has 10 edges, the maximum size of a matching in $G$ is $3$, and the maximum size of an induced matching in $G$ is $1$ (since the diameter of $G$ is two). So if we use two matchings, then we have to (and can) use $10-3\cdot 2=4$ induced matchings for the rest of the edges.  
\end{example}

\begin{figure}
\begin{center}
  \includegraphics[scale=0.7]{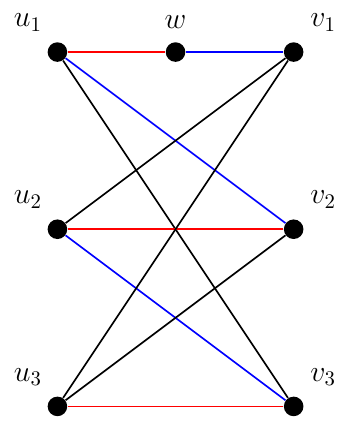} \hspace{10mm}
\end{center}
\caption{A sharp example for the $(1^2, 2^4)$-packing edge-coloring of subcubic graphs.}\label{ex-1}
\end{figure}

\begin{remark}
Our method works extremely well for connected subcubic graphs with more than $70$ vertices. We believe the proof can be extended to subcubic graphs with at most $70$ vertices by adding a few more pages case analysis. However, we do not plan to include it here.
\end{remark}
   


\section{Main Lemmas}

We may assume that $G$ is a connected cubic graph with more than $70$ vertices. If $G$ is a connected subcubic graph with at least one vertex of degree at most two, then we can turn $G$ into a cubic graph by adding edges and vertices. Let $G$ be a connected cubic graph with more than $70$ vertices and is not $(1^2,2^4)$-edge-colorable.

\begin{defn}
For each $M_1 \cup M_2$ maximum in $G$, let $H_{M_1 \cup M_2}$ (in short $H$) be the graph such that:

(1) $V(H) = E(G-M_1-M_2)$;

(2) $E(H) = \{e_1e_2 | d_G(e_1, e_2) \le 2, e_1, e_2 \in V(H)\}$.

\end{defn}

We take $M_1 \cup M_2$ maximum and its $H$ graph further satisfies:
\begin{itemize}
\item[(a)]  $H$ has the minimum number of edges;

\item[(b)] subject to (a), we require the number of $P_4$ (a path with three edges) in $G-M_1-M_2$ is minimized;

\item[(c)] subject to (a) and (b), we require
the number of triangles in $H$ is minimized;

\item[(d)] subject to (a), (b), and (c), we further require the number of paired $P_3$s that is joined by an edge in $M_1 \cup M_2$ through their middle vertex is minimized.
\end{itemize}
\vspace{2mm}

Our goal is to show that $H$ is $4$-colorable. Suppose not, i.e., $H$ is not $4$-colorable. Then it implies there is a 5-critical spanning subgraph $H_0$ of $H$. We recall a result of Kostochka and Yancey:

\begin{theorem}[Kostochka and Yancey~\cite{KY1}]\label{density}
If $G$ is a $k$-critical graph, then $$|E(G)| \ge (\frac{k}{2} - \frac{1}{k-1})|V(G)| - \frac{k(k-3)}{2(k-1)}.$$
\end{theorem}

By Theorem~\ref{density}, we know for our $5$-critical graph $H_0$,
$$|E(H_0)| \ge \frac{9}{4} |V(H_0)| - \frac{5}{4} \quad \text{and} \quad \sum\limits_{v \in V(H_0)} d_{H_0}(v) = 2 |E(H_0)| \ge 4.5 |V(H_0)| - 2.5.$$
Since $|V(H)| = |V(H_0)|$, $$\sum\limits_{v \in V(H)} (d_{H}(v) - 4.5) \ge \sum\limits_{v \in V(H_0)} (d_{H_0}(v) - 4.5) \ge -2.5.$$

We use the discharging method to obtain a contradiction. Let the initial charge of each vertex in $H$ be $\mu(v) = d_{H}(v) - 4.5.$ Our goal is to redistribute the charges so that the final charge $\mu^{\star}(v)$ for each $v \in V(H)$ satisfies 
\begin{equation}\label{charges}
\sum\limits_{v \in V(H)} \mu^{\star}(v) < -2.5.
\end{equation}

We need results on structures of each connected component of $\hat{G}:=G-M_1-M_2$. In the following proof, we use color black to denote an edge in $M_1 \cup M_2$ and use color red to denote an edge in $E(\hat{G})$. We usually put a $1$ ($2$) on a black edge to denote it belongs to $M_1$ ($M_2$).

\begin{lemma}\label{3-edges}
Each component of $\hat{G}$ has no cycle and has at most three edges.
\end{lemma}

\begin{proof}
We may assume $\hat{G}$ is connected and prove this statement in several steps.

\begin{figure}[ht]

\begin{center}
  \includegraphics[scale=0.8]{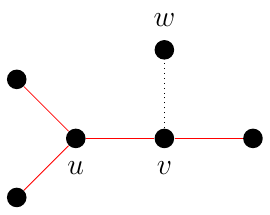} \hspace{15mm}
  \includegraphics[scale=0.8]{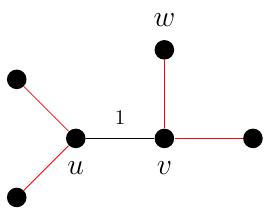} \hspace{15mm}
  \includegraphics[scale=0.8]{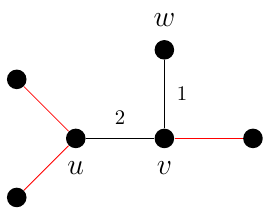}
\caption{No Configuration $C_1$ (on the left) and its proof.}\label{C_1}
\end{center}
\end{figure}


\begin{claim}\label{noc_1}
Let $C_1$ denote the configuration which is a tree with four edges and diameter three (See Figure~\ref{C_1}). There is no $C_1$ in $\hat{G}$ as a subgraph.
\end{claim}

\begin{proof}
We show the claim in two cases.

\textbf{Case 1:} The edge $vw$ is a red edge (See the middle graph in Figure~\ref{C_1}). Then we add $uv$ to $M_1$ to obtain an $M_1'$ such that $|M_1' \cup M_2| > |M_1 \cup M_2|$, which is a contradiction.

\textbf{Case 2:} The edge $vw$ is a black edge. Say $vw \in M_1$ (See the graph to the right in Figure~\ref{C_1}). Then we add $uv$ to $M_2$ to obtain an $M_2'$ such that $|M_1 \cup M_2'| > |M_1 \cup M_2|$, which is a contradiction.
\end{proof}

\noindent \textbf{Remark:} From this Claim, if there is a $3$-vertex $v$ in $\hat{G}$ then $N[v]$ must be a component in $\hat{G}$.

\begin{claim}
There is no cycle in $\hat{G}$.
\end{claim}

\begin{proof}

Suppose not, i.e., there is a cycle of length $k \ge 3$ in $\hat{G}$. Let the cycle be $v_1v_2 \cdots v_k$. By Claim~\ref{noc_1} and $k \ge 3$, the neighbour ($w_i$) of each $v_i$ that is not in the cycle must be connected to $v_i$ by a black edge, i.e., $w_iv_i \in M_1 \cup M_2$, $i \in [k]$ (See Figure~\ref{C_2}).

Let $i \in [k]$ and we may assume $v_iw_i \in M_1$. We must have $v_{i-1}w_{i-1} \in M_2$ and $v_{i+1}w_{i+1} \in M_2$, where $i+1$ and $i-1$ are computed after mod $k$. Suppose not, say $v_1w_1 \in M_1$ and $v_kw_k \in M_1$. Then we add $v_kv_1$ to $M_2$ to obtain $M_2'$ such that $|M_1 \cup M_2'| > |M_1 \cup M_2|$, which is a contradiction (See the middle graph in Figure~\ref{C_2}). Therefore, if we assume without loss of generality that $v_1w_1 \in M_1$, then we must have $v_2w_2 \in M_2, v_3w_3 \in M_1, \ldots, v_kw_k \in M_2$. In other words, the cycle must be even and the $v_iw_i$'s must alternate between in $M_1$ and in $M_2$.

We consider the graph $G' = G[M_1 \cup M_2]$. We know $\Delta(G') \le 2$ and therefore each non-trivial component of $G'$ is either a path or a cycle. We must have the case that $v_1w_1$ and $v_2w_2$ are in the same component of $G'$. Otherwise, say $v_1w_1 \in G'_1$ and $v_2w_2 \in G_2'$, we can switch $M_1$ edges with $M_2$ edges in $G'_1$ to obtain $M_1'$ and $M_2'$. Then we add $v_1v_2$ to $M_1'$ to obtain $M_1''$ such that $|M_1'' \cup M_2'| > |M_1 \cup M_2|$, which is a contradiction. Let the component $v_1w_1$ and $v_2w_2$ are located be called $G'_1$. However, we know $v_kw_k$ is not in the same component with $v_1w_1$ and $v_2w_2$, since there cannot be three leaves in $G'_1$. We switch $M_1$ edges with $M_2$ edges in $G'_1$ to obtain $M_1'$ and $M_2'$. Then we add $v_1v_2$ to $M_1'$ to obtain $M_1''$ such that $|M_1'' \cup M_2'| > |M_1 \cup M_2|$, which is a contradiction (See the graph to the right in Figure~\ref{C_2}). 
\end{proof} 

\begin{figure}[ht]

\begin{center}
  \includegraphics[scale=0.7]{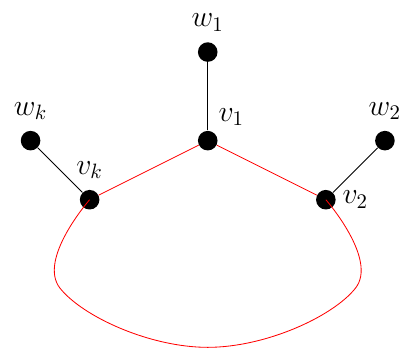} \hspace{5mm}
  \includegraphics[scale=0.7]{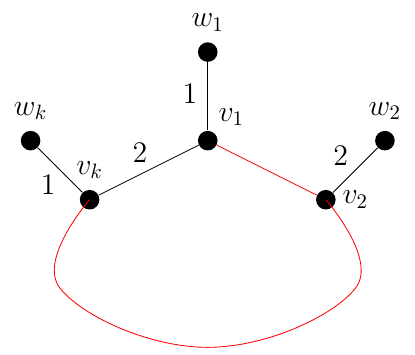} \hspace{5mm}
  \includegraphics[scale=0.7]{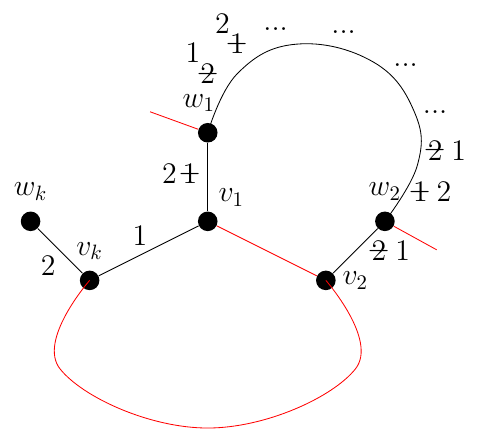}
\caption{No cycle (on the left) and its proof.}\label{C_2}
\end{center}
\end{figure}

\begin{claim}\label{nopath4}
There is no path with at least four edges in $\hat{G}$.
\end{claim}

\begin{proof}
Suppose not, i.e., there is a red path with four edges. Let $u_1u_2u_3u_4u_5$ be such a red path and the neighbours of $u_2,u_3,u_4$ outside of the path be $v_2, v_3, v_4$ respectively. By Claim~\ref{noc_1}, we know each of $u_2v_2, u_3v_3, u_4v_4$ must be black. Furthermore, we know $u_2v_2, u_3v_3, u_4v_4$ must alternates to be in $M_1$ and $M_2$. To see this, say both of $u_2v_2$ and $u_3v_3$ are in $M_1$, then we can add $u_2u_3$ to $M_2$ to obtain $M_2'$ such that $|M_1 \cup M_2'| > |M_1 \cup M_2|$, which is a contradiction. We may assume $u_2v_2 \in M_1, u_3v_3 \in M_2, u_4v_4 \in M_1$.

Then we consider $G' = G[M_1 \cup M_2]$. Since $\Delta(G') \le 2$, each non-trivial component of $G'$ is either a path or a cycle. We must have the case that $u_2v_2$ and $u_3v_3$ are in the same component of $G'$. Otherwise, say $u_2v_2 \in G'_1$ and $u_3v_3 \in G'_2$, then we can switch the $M_1$ edges with $M_2$ edges in $G'_1$ to obtain a new $M_1'$ and $M_2'$. Then we add $u_2u_3$ to $M_1'$ to obtain $M_1''$ such that $|M_1'' \cup M_2'| > |M_1 \cup M_2|$, which is a contradiction. Thus, $u_2v_2$ and $u_3v_3$ are in the same component of $G'$, say $G'_1$. However, we know $u_4v_4$ is not in $G'_1$, since there cannot be three leaves in $G'_1$. Then we switch the $M_1$ edges with the $M_2$ edges in $G'_1$ to obtain a new $M_1'$ and $M_2'$. Then we add $u_3u_4$ to $M_2'$ to obtain $M_2''$ such that $|M_1' \cup M_2''| > |M_1 \cup M_2|$, which is a contradiction (See Figure~\ref{C_3}).
\end{proof}
These three claims show the statement of the lemma.\end{proof}

\begin{figure}[ht]
\begin{center}
  \includegraphics[scale=0.8]{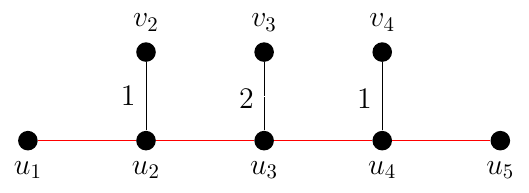} \hspace{10mm}
  \includegraphics[scale=0.8]{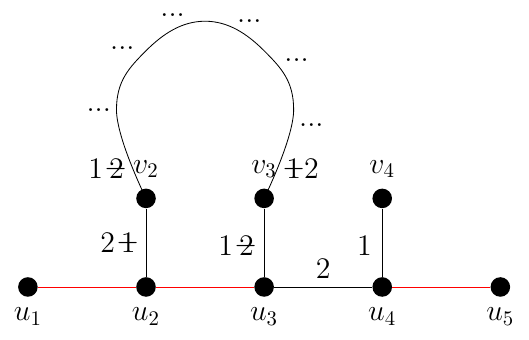}
\caption{No path of four edges and its proof.}\label{C_3}
\end{center}
\end{figure}

By Lemma~\ref{3-edges}, we know each component in $\hat{G}$ has no cycle and at most three edges. Therefore, a component can be a $K_{1,3}$, $P_4$ (a path with four vertices), $P_3$, or a $P_2$. We call them {\em basic components}.

\begin{lemma}\label{3stars}
Two $K_{1,3}$s in $\hat{G}$ cannot be joined by an edge in $M_1 \cup M_2$.
\end{lemma}

\begin{proof}
Suppose there are two $K_{1,3}$s joined by an edge in $M_1 \cup M_2$. Let $uu_1u_2u_3$ be a star with center $u$, $vv_1v_2v_3$ be a star with center $v$, and $u_1v_1$ be an edge in $M_2$. Let $N(u_1) = \{u,v_1, u_1'\}$ and $N(v_1) = \{v, u_1, v_1'\}$. Since $u_1v_1 \in M_2$, $u_1u_1', v_1v_1' \in M_1$. We remove $u_1v_1$ from $M_2$ and add $uu_1,vv_1$ to obtain a new matching $M_2'$ such that $|M_1 \cup M_2'| > |M_1 \cup M_2|$, which is a contradiction (See Figure~\ref{C_8}).
\end{proof}

\begin{figure}[ht]
\begin{center}
  \includegraphics[scale=0.7]{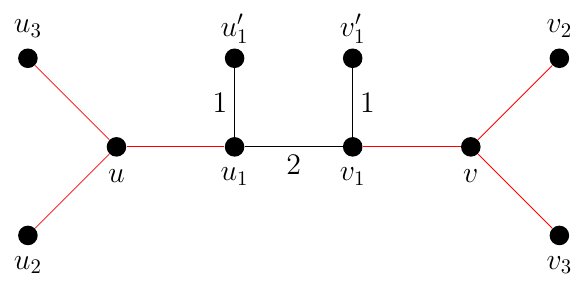} \hspace{10mm}
  \includegraphics[scale=0.7]{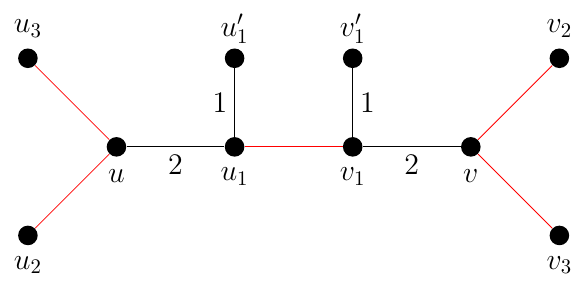}
\caption{No two $K_{1,3}$s joined by an edge in $M_1 \cup M_2$ and its proof.}\label{C_8}
\end{center}
\end{figure}

\begin{observation}\label{3pathmiddlevertices}
Let $u_1u_2u_3u_4$ be a $P_4$ in $\hat{G}$, where the neighbours of $u_2,u_3$ outside of the path are $v_2$ and $v_3$. By the proof of Claim~\ref{nopath4}, either $u_2v_2 \in M_1$ and $u_3v_3 \in M_2$ or $u_2v_2 \in M_2$ and $u_3v_3 \in M_1$. Furthermore, $u_2v_2$ and $u_3v_3$ must be in the same path-component of $G[M_1 \cup M_2]$.
\end{observation}

By Observation~\ref{3pathmiddlevertices}, the path component in $G[M_1 \cup M_2]$ containing $u_2u_2'$ and $u_3u_3'$ is a path with even number of edges. Let the path be $w_1\cdots w_k$, where $w_1 = u_2'$ and $w_k = u_3'$. Furthermore, we know $\Delta(G[M_1 \cup M_2]) \le 2$ and therefore each $w_i$ has a neighbour $w_i'$ such that $w_iw_i' \in \hat{G}$. 

\begin{lemma}\label{3pathmiddlestructure}
Let $u_1u_2u_3u_4$ be a $P_4$ in $\hat{G}$ and $u_2w_1\cdots w_ku_3$ be a path in $M_1\cup M_2$. Let $w_i'$ be the neighbor of $w_i$ not on the path, and $N(w_i') = \{w_i, x_i, y_i\}$, where $i \in \{1, \ldots, k\}$. Then $w_i'x_i, w_i'y_i \in M_1 \cup M_2$. 
\end{lemma}

\begin{proof}
Suppose there is an $i \in \{1, \ldots, k\}$ such that $w_i'x_i \in \hat{G}$. By symmetry, we may assume that $w_{i-1}w_i \in M_1$ and $w_iw_{i+1} \in M_2$. 
If $w_i'y_i \in M_1 \cup M_2$, we can assume $w_i'y_i \in M_2$.  We delete $w_iw_{i+1}$ from $M_2$, add $w_iw_i', u_2u_3$ to $M_1$, and switch the $M_1, M_2$ edges in the path $w_1\cdots w_i$. This will give us two new disjoint matchings $M_1', M_2'$ with $|M_1' \cup M_2'| > |M_1 \cup M_2|$, which is a contradiction (See Figure~\ref{C_4} left and middle picture). 
\end{proof}

\begin{figure}[ht]
\begin{center}
    \includegraphics[scale=0.6]{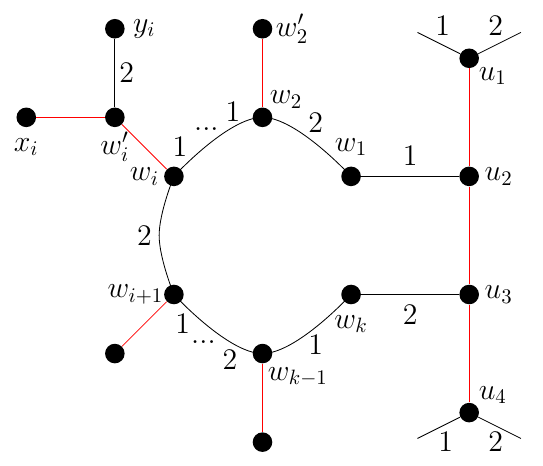} 
    \includegraphics[scale=0.6]{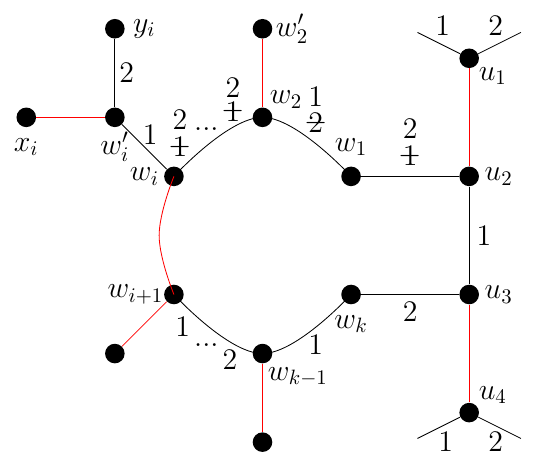} 
    \includegraphics[scale=0.6]{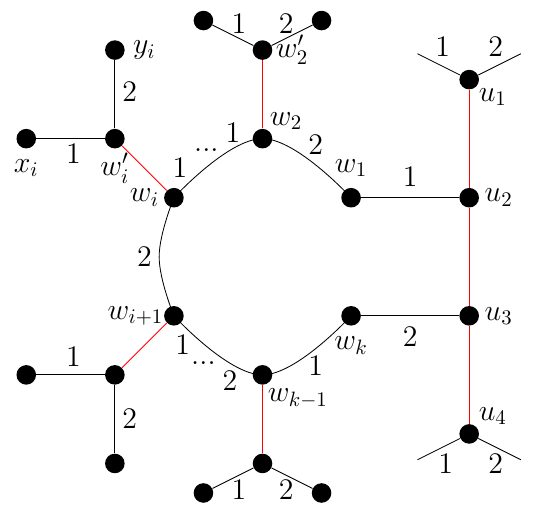}
\caption{$P_4$ and its middle cycle.}\label{C_4}
\end{center}
\end{figure}

\begin{lemma}\label{3star3path}
A $K_{1,3}$ in $\hat{G}$ cannot be joined to a $P_4$ in $\hat{G}$ by an edge in $M_1 \cup M_2$.
\end{lemma}

\begin{proof}
Suppose not. Let $vv_1v_2v_3$ be a star with center $v$, $u_1u_2u_3u_4$ be a path. Since they can be joined by an edge in two non-symmetric ways, we have two cases.

\textbf{Case 1:} $v_1u_2 \in M_1 \cup M_2$. Say $v_1u_2 \in M_1$. Then we may assume $u_3u_3' \in M_2$, since otherwise we can add $u_2u_3$ to $M_2$ to obtain a new matching $M_2'$ such that $|M_1 \cup M_2'| > |M_1 \cup M_2|$, which is a contradiction. We know $v_1v_1'$ must be in $M_2$ as well. Therefore, we can delete $v_1u_2$ and add $vv_1, u_2u_3$ to $M_1$ to obtain a new matching $M_1'$. It is a contradiction since $|M_1' \cup M_2| > |M_1 \cup M_2|$ (See Figure~\ref{C_5}).

\textbf{Case 2:} $v_1u_1 \in M_1 \cup M_2$. Say $v_1u_1 \in M_2$. Then $v_1v_1', u_1u_1'$ must be in $M_1$. We must have $u_2u_2' \in M_2$ since otherwise we can delete $v_1u_1$ and add $u_1u_2, vv_1$ to $M_2$ to obtain a new matching $M_2'$ such that $|M_1 \cup M_2'| > |M_1 \cup M_2|$, which is a contradiction. We must also have $u_3u_3' \in M_1$ since otherwise we can add $u_2u_3$ to $M_1$ to obtain a new matching $M_1'$. It is a contradiction since $|M_1' \cup M_2| > |M_1 \cup M_2|$. 

Then we consider $G' = G[M_1 \cup M_2]$. Since $\Delta(G') \le 2$, each non-trivial component of $G'$ is either a path or a cycle. By Observation~\ref{3pathmiddlevertices}, $u_2u_2'$ and $u_3u_3'$ are in the same component of $G'$. Let the component $u_2u_2'$ and $u_3u_3'$ was located be called $G_1'$. We must have $u_1u_1'$ and $v_1v_1'$ are not in $G_1'$ since $G_1'$ is a path component in $G'$ and $u_2,u_3$ are both leaf vertices in $G_1'$. Therefore, we can switch the $M_1$ edges with $M_2$ edges in $G_1'$ and add $u_1u_2,vv_1$ to $M_2$ to obtain a new $M_2'$ such that $|M_1 \cup M_2'| > |M_1 \cup M_2|$, which is a contradiction (See Figure~\ref{C_5}).
\end{proof}


\begin{figure}[ht]
\begin{center}
  \includegraphics[scale=0.7]{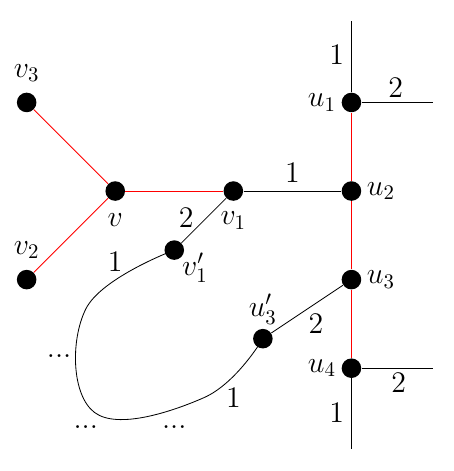} \hspace{20mm}
  \includegraphics[scale=0.7]{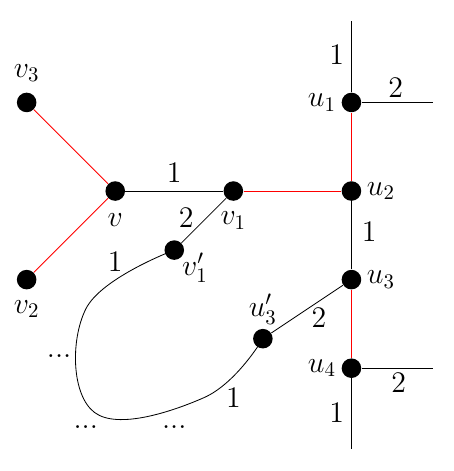}
    \includegraphics[scale=0.6]{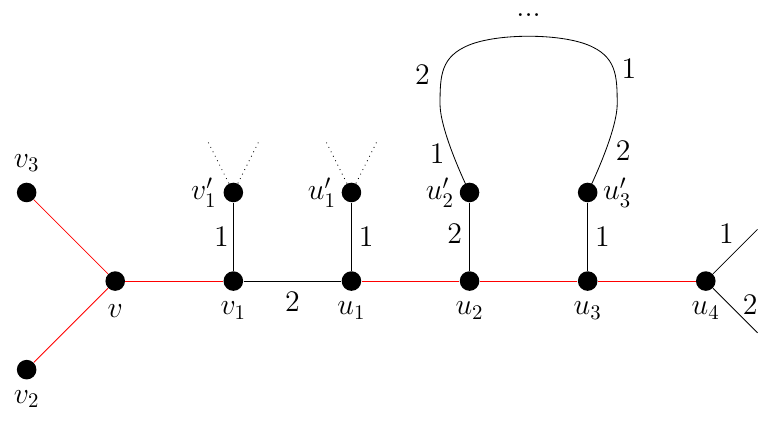} \hspace{0mm}
  \includegraphics[scale=0.6]{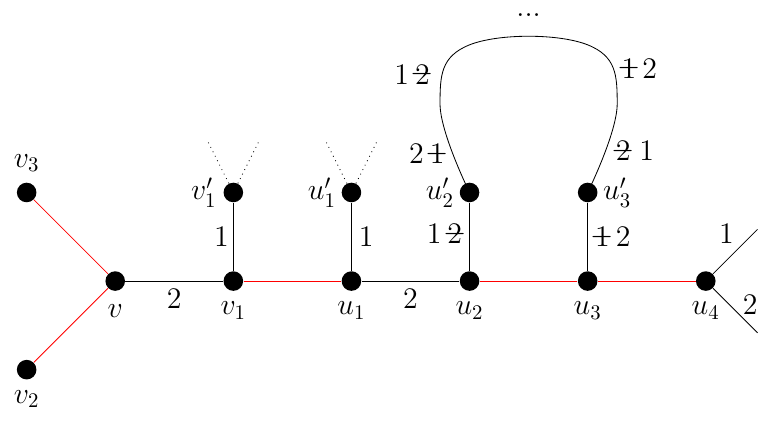}
\caption{No $K_{1,3}$ and $P_4$ joined by an edge in $M_1 \cup M_2$ and its proof.}\label{C_5}
\end{center}
\end{figure}

By Lemma~\ref{3pathmiddlestructure} and~\ref{3star3path}, we already knew the structure in the middle two vertices of a $P_4$ $u_1u_2u_3u_4$ in $\hat{G}$ (See Figure~\ref{C_7} the right picture). Lemma~\ref{3star3path} also tells us some structures on the endpoint of the $P_4$. We can further analyse the structure outside of $u_1$ and $u_4$. 

\begin{lemma}\label{3path2path}
A $P_4$ in $\hat{G}$  cannot be joined to the middle vertex of a $P_3$ in $\hat{G}$  by an edge in $M_1 \cup M_2$.
\end{lemma}

\begin{proof}
Suppose not, a $P_4=u_1u_2u_3u_4$ is joined to a $P_3=v_1vv_2$ by an edge in $M_1 \cup M_2$.   


By Observation~\ref{3pathmiddlevertices}, we assume that $u_1v \in M_1\cup M_2$. Say $u_1v \in M_1, u_1u_1' \in M_2$. Let $N(u_2) = \{u_1, u_3, u_2'\}$, $N(u_3) = \{u_2, u_4, u_3'\}$, and $N(u_1) = \{u_2, v, u_1'\}$. By Observation~\ref{3pathmiddlevertices}, either $u_2u_2' \in M_1$ and $u_3u_3' \in M_2$ or $u_2u_2' \in M_2$ and $u_3u_3' \in M_1$. Furthermore, $u_2u_2'$ and $u_3u_3'$ are in the same path component of $G[M_1 \cup M_2]$ (say this component is $G_1'$). Since $v$ is a degree one vertex in $G[M_1 \cup M_2]$, $u_1u_1', u_1v$ are not in $G_1'$. We can assume $u_2u_2' \in M_2$ and $u_3u_3' \in M_1$ since if $u_2u_2' \in M_1$ and $u_3u_3' \in M_2$ then we switch the $M_1,M_2$ edges in $G_1'$. This will again give us the case $u_2u_2' \in M_2$ and $u_3u_3' \in M_1$. We delete $u_1v$ from $M_1$ and add $u_1u_2$ to $M_1$. The number of edges in $H$ is dropped by two, which is a contradiction (See Figure~\ref{C_6}).
\end{proof}

\begin{figure}[ht]
\begin{center}
    \includegraphics[scale=0.6]{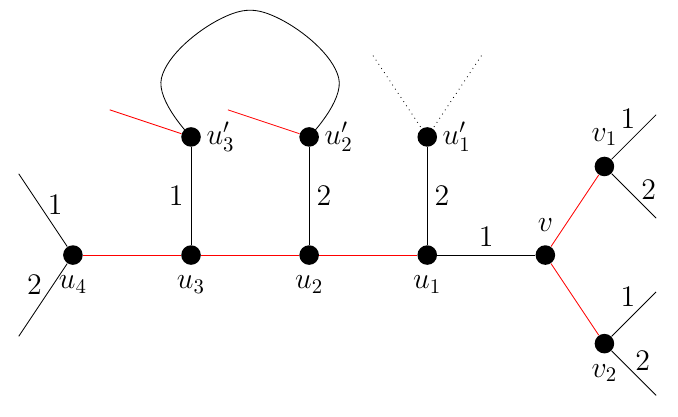} \hspace{0mm}
  \includegraphics[scale=0.6]{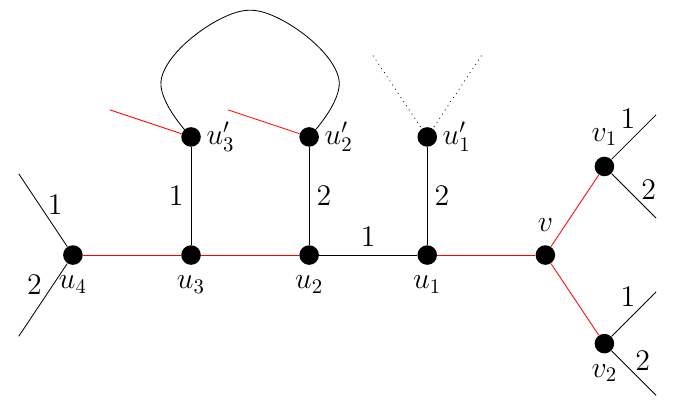}
\caption{No $P_4$ can join to the middle vertex of a $P_3$ by an edge in $M_1 \cup M_2$ and its proof.}\label{C_6}
\end{center}
\end{figure}

Now we set up for Lemma~\ref{3pathoutside} -~\ref{p43cycle}. Let $u_1u_2u_3u_4$ be a $P_4$ in $\hat{G}$.  Let $N(u_1) = \{u_2, v_1, z_1\}$. By Lemma~\ref{3-edges}, $u_1v_1, u_1z_1 \in M_1 \cup M_2$. We may assume $u_1v_1 \in M_1$ and $u_1z_1 \in M_2$. Let $N(v_1) = \{u_1, v_2, v_3\}$ and $N(z_1) = \{u_1, z_2, z_3\}$. By Lemma~\ref{3path2path}, we cannot have both $v_1v_2, v_1v_3 \in \hat{G}$. Thus, we may assume $v_1v_2 \in \hat{G}$. Similarly, we may assume $z_1z_2 \in \hat{G}$. Since $u_1v_1 \in M_1$ and $u_1z_1 \in M_2$, we know $v_1v_3 \in M_2$ and $z_1z_3 \in M_1$. Let $N(v_2) = \{v_1, v_4, v_5\}$, $N(v_3) = \{v_1, v_6, v_7\}$, and $N(v_6) = \{v_3, v_8, v_9\}$. Then we have the following lemma. 

\begin{figure}[ht]
\begin{center}
  \includegraphics[scale=0.6]{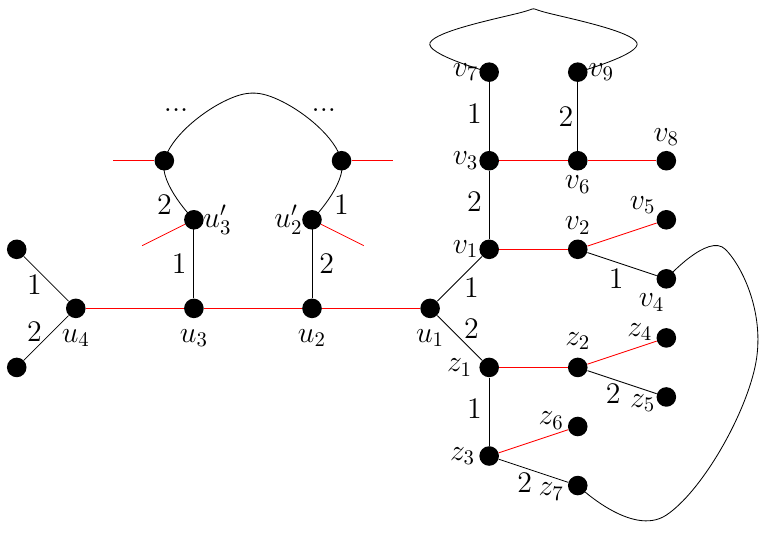} \hspace{10mm}
  \includegraphics[scale=0.6]{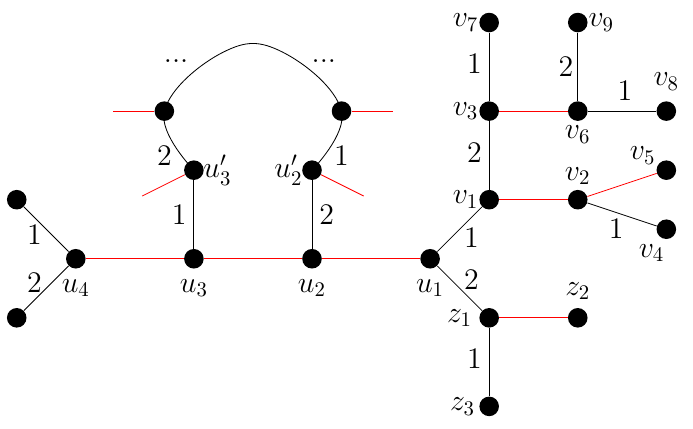} 
\caption{The structure outside a $P_4$.}\label{C_7}
\end{center}
\end{figure}

\begin{lemma}\label{3pathoutside}
Up to symmetry, $v_3v_6, v_2v_5 \in \hat{G}$, $v_3v_7 \in M_1$, $v_2v_4 \in M_1$, $v_6v_8 \in M_1$ and $v_6v_9 \in M_2$. (See Figure~\ref{C_7} right picture)
\end{lemma}

\begin{proof}
By Lemma~\ref{3star3path}, $v_2v_5$ and $v_2v_4$ cannot be both in $\hat{G}$. By Observation~\ref{3pathmiddlevertices}, we may assume $u_2u_2' \in M_2$ and $u_3u_3' \in M_1$. If $v_2v_4, v_2v_5 \in M_1 \cup M_2$, then we delete $u_1v_1$ from $M_1$ and add $u_1u_2$ to $M_1$. The number of edges in $H$ will not increase (even when both $v_3v_6, v_3v_7 \in \hat{G}$) and the number of $P_4$s in $\hat{G}$ is dropped by one, which is a contradiction. Thus, we may assume $v_2v_4 \in M_1 \cup M_2$ and $v_2v_5 \in \hat{G}$. We claim $v_3v_7 \in M_1 \cup M_2$, since otherwise we can delete $v_1v_3$ from $M_2$ and $u_1v_1$ from $M_1$, and add $u_1u_2$ to $M_1$ and $v_1v_2, v_1v_3$ to $M_1,M_2$ or $M_2, M_1$ (depending on whether $v_2v_4$ is in $M_1$ or $M_2$). This is a contradiction since the new matchings $M_1', M_2'$ satisfy $|M_1' \cup M_2'| > |M_1 \cup M_2|$. Thus, $v_3v_7 \in M_1$. By a similar switching argument, we must have $v_2v_4 \in M_1$. Furthermore, $v_2v_4, v_3v_7$ are in the same component in $G[M_1 \cup M_2]$, since otherwise we switch edges between $M_1, M_2$ for the component where $v_2v_4$ is located in $G[M_1 \cup M_2]$. Then we delete $u_1v_1$ from $M_1$ and add $u_1u_2, v_1v_2$ to $M_1$ to obtain two new matchings $M_1', M_2'$ satisfy $|M_1' \cup M_2'| > |M_1 \cup M_2|$, which is a contradiction. 

At last, we show $v_6v_8, v_6v_9 \in M_1 \cup M_2$. If both $v_6v_8, v_6v_9 \notin M_1 \cup M_2$, then we delete $v_1v_3$ from $M_2$ and add $v_3v_6, v_1v_2$ to $M_2$ to obtain a new matching $M_2'$ such that $|M_1 \cup M_2'| > |M_1 \cup M_2|$, which is a contradiction. Thus, we may assume one of $v_6v_8$ and $v_6v_9$ is in $M_1 \cup M_2$, say $v_6v_9 \in M_1 \cup M_2$ and $v_6v_8 \notin M_1 \cup M_2$. By a similar switching argument, we know $v_6v_9 \in M_2$. Furthermore, $v_6v_9$ must be in the same component with $v_1v_3$ and $v_2v_4$ in $G[M_1 \cup M_2]$, since otherwise we can switch matching edges to make $v_6v_9 \in M_1$, which is a contradiction. Let $N(z_2) = \{z_1, z_4, z_5\}$ and $N(z_3) = \{z_1, z_6, z_7\}$. By symmetry, exactly one of $z_2z_4$ and $z_2z_5$ is in $\hat{G}$ and exactly one of $z_3z_6$ and $z_3z_7$ is in $\hat{G}$. Say $z_2z_4,z_3z_6 \in \hat{G}$. By symmetry again, we know $z_2z_5, z_3z_7 \in M_2$ and $z_2z_5$ is in the same component with $z_3z_7$ in $G[M_1 \cup M_2]$ (say $G_1'$). However, this is a contradiction since we have three leaves $z_2, v_2, v_6$ in the same component $G_1'$ such that $\Delta(G_1') \le 2$ (See Figure~\ref{C_7} left picture).
\end{proof}

We consider $G':=G[M_1 \cup M_2]$ and claim $v_1v_3$ and $v_2v_4$ must be in the same component.

\begin{lemma}\label{3pathoutside-2}
$v_1v_3$ and $v_2v_4$ must be in the same component of $G'$ and $v_1v_2v_4 \cdots v_7v_3v_1$ forms an odd cycle.
\end{lemma}
\begin{proof}
Suppose $v_1v_3$ and $v_2v_4$ are not in the same component of $G'$. Say $v_2v_4$ is in the component $G_1'$. We switch the $M_1$ and $M_2$ edges in $G_1'$ and then delete $u_1v_1$ from $M_1$ and add $u_1u_2, v_1v_2$ to $M_1$ to obtain new matchings $M_1'$ and $M_2'$ such that $|M_1' \cup M_2'| > |M_1 \cup M_2|$, which is a contradiction.
\end{proof}

By Lemma~\ref{3pathoutside} and~\ref{3pathoutside-2}, the upper right corner of Figure~\ref{C_7} (right picture) shows a fixed structure whenever we have a $P_4$ in $\hat{G}$: the endpoint $u_1$ of the $P_4$ is adjacent to three odd cycles $u_2u_2'\ldots u_3'u_3$ (through red edge $u_1u_2$), $v_1v_3v_7\ldots v_4v_2$ (through black edge $u_1v_1$), and $z_1z_2\dots z_3$ (through black edge $u_1z_1$), such that each of the odd cycles has exactly one red edge ($u_2u_3, v_1v_2$ and $z_1z_2$, respectively), and the edges incident to vertices on the cycles but outside of the cycles are all red. The structure is symmetry if we interchange the colors among the three edges $u_1v_1, u_1u_2, u_1z_1$ incident to $u_1$. We shall call $u_1$ to be a {\em central vertex}. Since we may interchange the colors between $u_2u_3$ and any black edge on the cycle $u_2u_2'\ldots u_3'u_3$ to obtain a new red $P_4$, each vertex incident to the odd cycles is a central vertex. Thus we obtain the structure of $G$ if there is a red $P_4$ (See Figure~\ref{C_7.5}). Furthermore, we show there is no $P_4$. 


\begin{figure}[ht]
\begin{center}
  \includegraphics[scale=0.38]{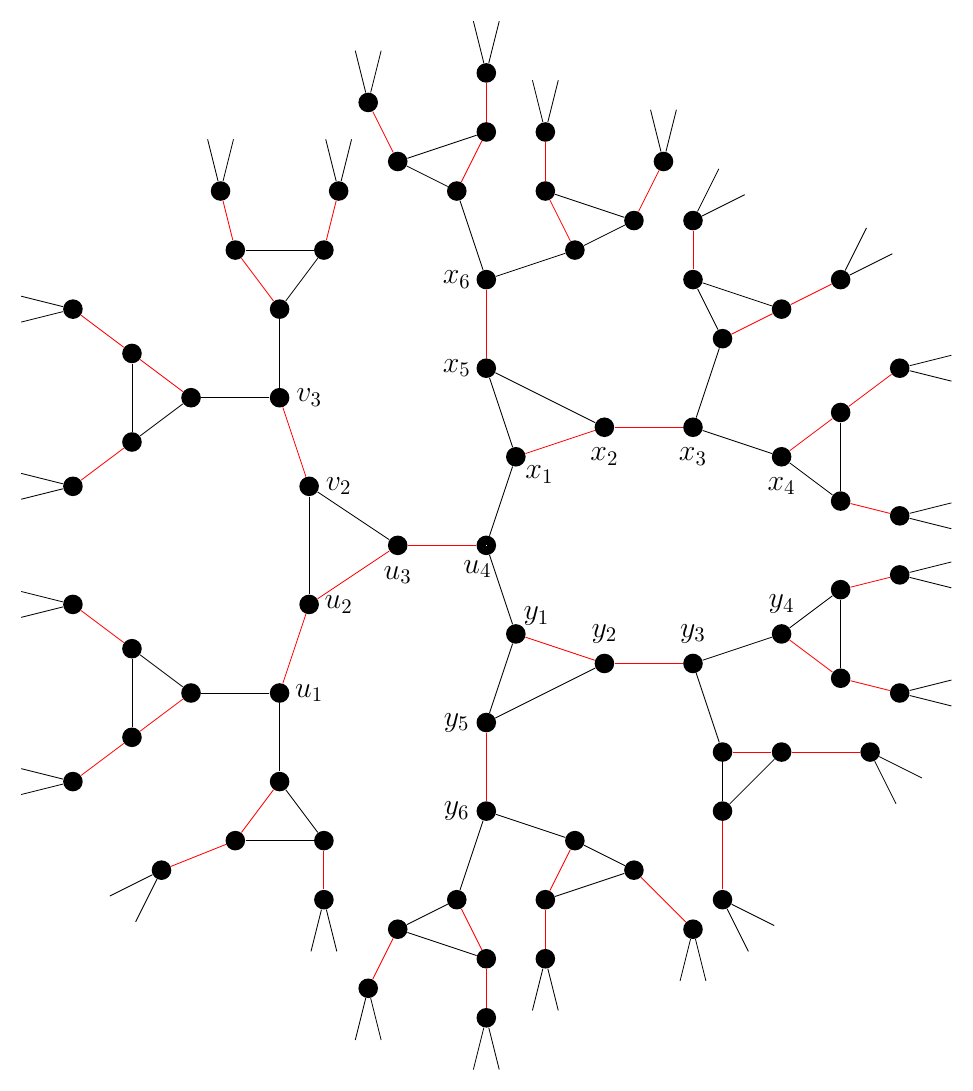} 
\caption{Structure of the graph if there is a $P_4$.}\label{C_7.5}
\end{center}
\end{figure}

\begin{lemma}\label{p43cycle}
There is no $P_4$.
\end{lemma}
\begin{proof}
We count the number of red edges (called $n_r$) and the number of black edges (called $n_b$) incident with the central vertices. Clearly, $n_b=2n_r$, since each central vertex is incident with two black edges and one red edge.  On the other hand, each involved odd cycle $C$ gives at least $|C|-1$ red edges and at most one black edge. It follows that $n_b\le \sum_C 1$ and $n_r\ge \sum_C (|C|-1)$. Therefore, we have $$\sum_C 1\ge n_b=2n_r\ge 2\sum_C (|C|-1),$$
which is a contradiction as $|C|\ge 3$.
\end{proof}

We now turn our attention to $P_3$s and $P_2$s. We first show that if two $P_3$s are joined by an edge in $M_1 \cup M_2$ via the middle vertex then the graph is determined.

\begin{figure}[ht]
\begin{center}
  \includegraphics[scale=0.55]{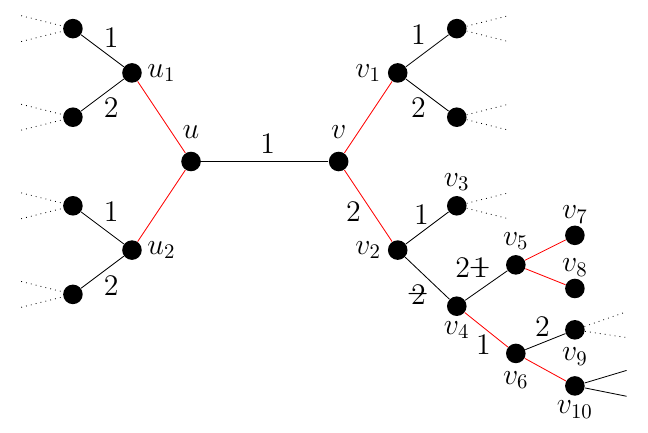}
  \includegraphics[scale=0.55]{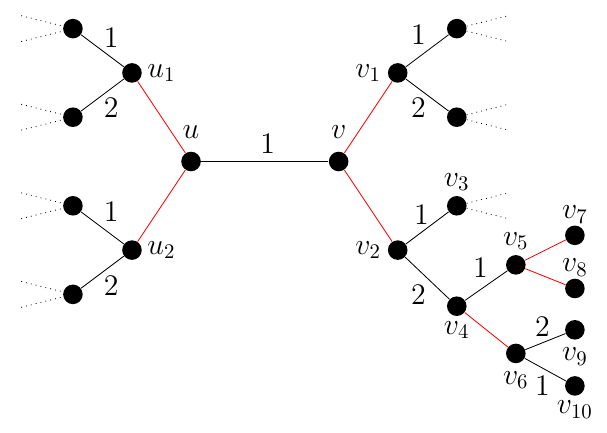}
\caption{Two $P_3$s cannot be joined by an edge in $M_1 \cup M_2$ via the middle vertex.}\label{C_9}
\end{center}
\end{figure}

\begin{lemma}\label{two2pathsmiddle}
If there is a pair of $P_3$s in $\hat{G}$ whose middle vertices are adjacent by an edge in $M_1 \cup M_2$, then the graph $\hat{G}$ contains only one such pair and the structure of the graph $G$ is fixed.
\end{lemma}

\begin{proof}
Suppose $v_1vv_2$ and $u_1uu_2$ are two $P_3$s in $\hat{G}$ such that $uv \in M_1 \cup M_2$.  Let $N(v_2) = \{v, v_3, v_4\}$ and $N(v_4) = \{v_2, v_5, v_6\}$. We may assume $uv, v_2v_3 \in M_1$ and $v_2v_4 \in M_2$ (See Figure~\ref{C_9}). 

\begin{claim}\label{v4v5v6}
The edges $v_4v_5$ and $v_4v_6$ cannot be both in $\hat{G}$. 
\end{claim} 

\begin{proof}
Suppose both $v_4v_5$ and $v_4v_6$ are in $\hat{G}$. We delete $v_2v_4$ from $M_2$ and add $vv_2$ to $M_2$. The number of edges in $H$ is dropped two, which is a contradiction.
\end{proof}

Since $v_2v_4 \in M_2$, $v_4v_5, v_4v_6  \notin M_2$. By Claim~\ref{v4v5v6}, we may assume $v_4v_5 \in M_1$ and $v_4v_6 \in \hat{G}$. Let $N(v_5) = \{v_4, v_7, v_8\}$ and $N(v_6) = \{v_4, v_9, v_{10}\}$.

\begin{claim}
$v_6$ cannot be the center of a $K_{1,3}$. Furthermore, $v_6v_9$ or $v_6v_{10}$ must be in $M_2$.
\end{claim}

\begin{proof}
Suppose $v_6$ is the center of a $K_{1,3}$. Since $v_2v_4 \in M_2$, $v_4v_5 \in M_1$. We delete $v_2v_4$ from $M_2$ and add $vv_2, v_4v_6$ to $M_2$ to obtain a new matching $M_2'$ such that $|M_1 \cup M_2'| > |M_1 \cup M_2|$. This is a contradiction. By a similar switching argument, we know $v_6v_9$ or $v_6v_{10}$ must be in $M_2$.
\end{proof}

Therefore, either $v_6v_9 \in M_2$ and $v_6v_{10} \in \hat{G}$ or $v_6v_9 \in M_2$ and $v_6v_{10} \in M_1$. We show that only the latter case is possible and both of $v_5v_7,v_5v_8 \in \hat{G}$ (See Figure~\ref{C_9} right picture).

\begin{claim}
We must have $v_6v_9 \in M_2$ and $v_6v_{10} \in M_1$. Furthermore, $v_5v_7,v_5v_8 \in \hat{G}$.
\end{claim}

\begin{proof}
Suppose $v_6v_9 \in M_2$ and $v_6v_{10} \in \hat{G}$. We first show one of $v_5v_7$ and $v_5v_8$ must be in $M_2$. Suppose both of $v_5v_7$ and $v_5v_8$ are in $\hat{G}$. We delete $v_2v_4$ from $M_2$ and $v_4v_5$ from $M_1$, and then add $v_4v_6$ to $M_1$ and $vv_2, v_4v_5$ to $M_2$. This is a contradiction since the new $M_1'$ and $M_2'$ satisfy $|M_1' \cup M_2'| > |M_1 \cup M_2|$ (See Figure~\ref{C_9}). Thus, we may assume $v_5v_7 \in M_2$ and $v_5v_8 \in \hat{G}$. We delete $v_2v_4$ from $M_2$ and add $vv_2$ to $M_2$. The number of edges in $H$ does not change but the number of triangles in $H$ is dropped by one, which is a contradiction. Thus, we must have $v_6v_9 \in M_2$ and $v_6v_{10} \in M_1$.

We next show $v_5v_7,v_5v_8 \in \hat{G}$. Suppose one of $v_5v_7$ and $v_5v_8$ is in $M_2$. We delete $v_2v_4$ from $M_2$ and add $vv_2$ to $M_2$. The number of edges in $H$ is dropped by one, which is a contradiction. 
\end{proof}

By symmetry, the structure outside of $v_3$ must be the same to the structure outside of $v_4$. By symmetry again, the structure outside of $v_1,u_1,u_2$ must be the same with the structure outside of $v_2$. We call each of $v_1,v_2,u_1,u_2$ a {\em special vertex}. Since we can switch the edges $vv_2$ and $v_2v_4$ (then $v_4v_5$ becomes the central edge of two adjacent $P_3$s), we know the structure outside of $v_7,v_8,v_9,v_{10}$. By applying similar argument to $v_3$ and other vertices that are symmetric to $v_3$, we can determine the structure of the graph (See Figure~\ref{C_9} right picture; vertices are renamed in order to better show the structure ). 

At last, we show that there can be at most one pair of $P_3$s in the graph. Suppose $x_1u_1w_1$ and $y_1v_1z_1$ is a pair of $P_3$s linked by $u_1v_1$ and there is another pair of $P_3$s $x_1'u_1'w_1'$ and $y_1'v_1'z_1'$ linked by $u_1'v_1'$ in the graph. We define the distance between two such pairs to be the distance between their central edges $u_1v_1$ and $u_1'v_1'$. By previous Claims (See Figure~\ref{C_9} right picture), we know the distance is at least four. We may also assume the shortest path between $u_1v_1$ and $u_1'v_1'$ go through the edge $z_1v_6$. We switch the edges $v_1z_1$ and $z_1v_6$. Then the distance between two pairs of $P_3$s is dropped by one. This switching process can be applied in a sequence so that the distance between two pairs of $P_3$s is dropped to be below four. This is a contradiction with the fact that the distance must be at least four (See Figure~\ref{C_9}).
\end{proof}

We may assume that two $P_3$s in $\hat{G}$ cannot be joined by an edge in $M_1 \cup M_2$ via the middle vertex. However, they can still be joined in the types of end-end and end-middle vertices. We next show that an end vertex of a $P_3$ cannot be joined to two other $P_3$s by edges in $M_1 \cup M_2$ via the middle vertices.

\begin{figure}[ht]
\begin{center}
  \includegraphics[scale=0.7]{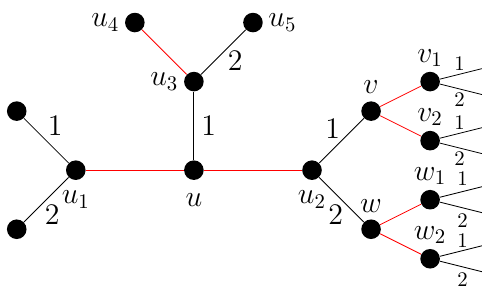}
  \includegraphics[scale=0.6]{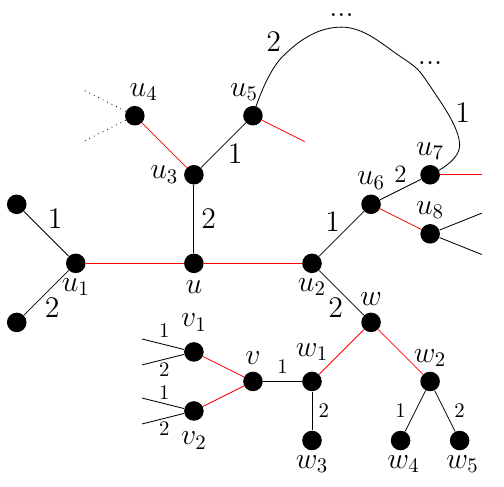}
  \includegraphics[scale=0.5]{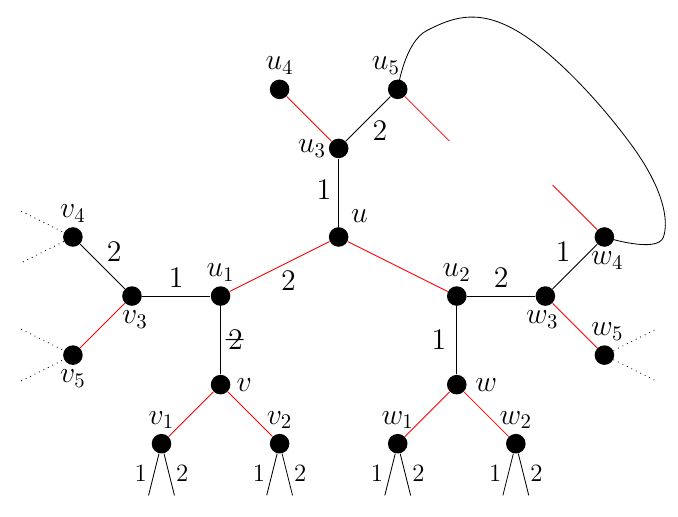} 
\caption{Three 2-paths.}\label{C_10}
\end{center}
\end{figure}

\begin{lemma}\label{three2-paths}
If an end vertex of a $P_3$ in $\hat{G}$ is already joinned to the middle vertex of a $P_3$ in $\hat{G}$ by an edge in $M_1 \cup M_2$, then this end vertex cannot be joined with another $P_3$ in $\hat{G}$ by an edge in $M_1 \cup M_2$.
\end{lemma}
\begin{proof}
Let $u_1uu_2$, $v_1vv_2$, and $w_1ww_2$ be three $P_3$s such that $u_2v \in M_1$. Let $N(u) = \{u_1, u_2, u_3\}$ and $N(u_3) = \{u, u_4, u_5\}$. We show $u_2$ cannot join to another $P_3$.

\textbf{Case 1:} $u_2$ is joined to $w_1ww_2$ by the edge $u_2w \in M_2$. We may assume $uu_3, u_2v \in M_1$, $u_3u_5, u_2w \in M_2$, and $u_3u_4 \in \hat{G}$. We delete $u_2w$ from $M_2$ and add $uu_2$ to $M_2$. This is a contradiction since the number of edges in $H$ is dropped by one (See Figure~\ref{C_10} left picture).

\textbf{Case 2:} $u_2$ is joined to $w_1ww_2$ by the edge $u_2w_1 \in M_2$. Let $N(w_1) = \{u_2, w, w_4\}$, $N(w) = \{w_1, w_2, w_3\}$. We know $w_1w_4 \in M_1$. We claim $uu_3 \in M_1$, since if $uu_3 \in M_2$ then we delete $vu_2$ from $M_1$ and add $uu_2$ to $M_1$. The number of edges in $H$ is dropped by one and it is a contradiction. By the same reason, $uu_3$ must be in the same component with $vu_2$ in $G[M_1 \cup M_2]$. Since $w, v, u$ are three degree one vertices in $G[M_1 \cup M_2]$, they cannot be in the same component in $G[M_1 \cup M_2]$. Therefore, we switch the matching edges between in $M_1$ and $M_2$ for the component $ww_3$ is located, and then delete $w_1u_2$ from $M_2$ and add $uu_2, ww_1$ to $M_2$ to obtain a new $M_2'$. This is a contradiction since $|M_1 \cup M_2'| > |M_1 \cup M_2|$ (See Figure~\ref{C_final} right picture).
\end{proof}

\begin{figure}[ht]
\begin{center}
  \includegraphics[scale=0.5]{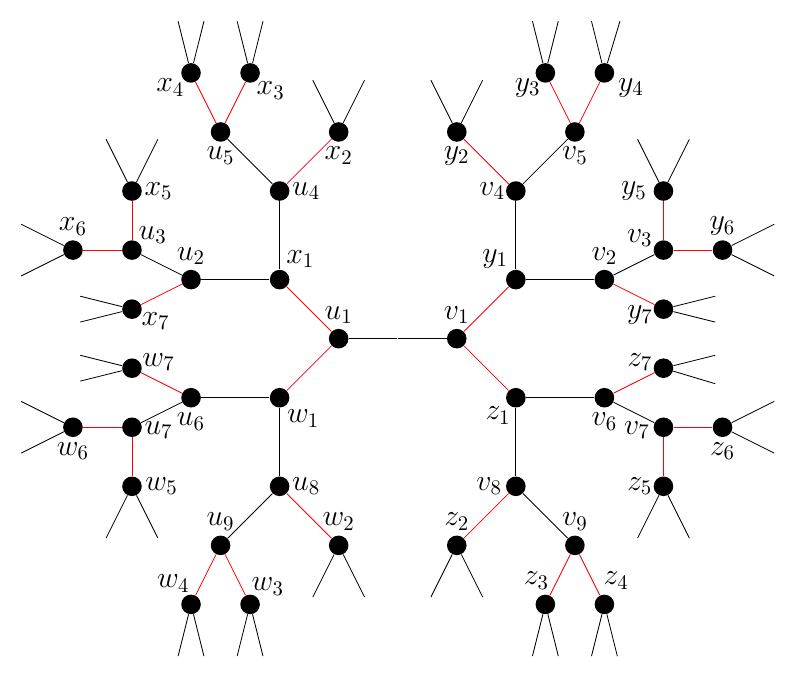} \hspace{10mm}
    \includegraphics[scale=0.7]{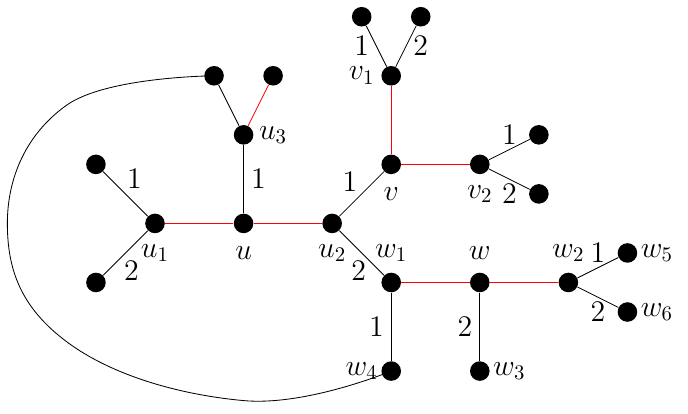}
\caption{Structure of the graph if there are $P_3$s connected by their middle vertices.}\label{C_final}
\end{center}
\end{figure}

Let $v_1vv_2$ be a $P_3$ in $\hat{G}$ and $N(v) = \{v_1, v_2, w_1\}$. We know $v_1vv_2$ cannot be adjacent to another $P_3$ in $\hat{G}$  by an edge in $M_1 \cup M_2$ via the middle vertices. We also know $v$ cannot be adjacent to a $P_4$ in $\hat{G}$. Thus, $v$ is either adjacent to a $K_{1,3}$, a $P_2$, or a $P_3$ in $\hat{G}$ via the endpoint by an edge in $M_1 \cup M_2$. We focus on the case when $v$ is adjacent to a $P_3=w_1ww_2$ in $\hat{G}$  via the endpoint, say $w_1$. Let $N(w_1) = \{v, w, w_3\}$ and $N(w) = \{w_1, w_2, u_2\}$. Similarly, $w$ is either adjacent to a $K_{1,3}$, $P_2$, or a $P_3$ in $\hat{G}$ via the endpoint by an edge in $M_1 \cup M_2$. We show that the latter case is impossible.

\begin{lemma}\label{three2-paths-2}
Let $v_1vv_2$ and $w_1ww_2$ be $P_3$s in $\hat{G}$ such that $vw_1\in M_1\cup M_2$. Then $w$ cannot be adjacent to a $P_3$  in $\hat{G}$ by an edge in $M_1 \cup M_2$.
\end{lemma}

\begin{proof}
Suppose $w$ is adjacent to a $P_3=u_1uu_2$ in $\hat{G}$ in $M_1 \cup M_2$.  By Lemma~\ref{two2pathsmiddle},  we may assume that $wu_2\in M_1\cup M_2$. Let $N(u) = \{u_1, u_2, u_3\}$, $N(u_2) = \{u, w, u_6\}$, $N(u_3) = \{u, u_4, u_5\}$, and $N(u_6) = \{u_2, u_7, u_8\}$ (See Figure~\ref{C_10} middle picture). We may assume $uu_3, u_2w \in M_2$ and $u_2u_6 \in M_1$. By Lemma~\ref{two2pathsmiddle}, we may assume $u_3u_4 \in \hat{G}$ and $u_3u_5 \in M_1$. By Lemma~\ref{three2-paths}, we may assume $u_6u_7 \in M_2$ and $u_6u_8 \in \hat{G}$. We consider $G' = G[M_1 \cup M_2]$. We know $uu_3$ and $u_2u_6$ must be in the same component, say $G_1'$ of $G'$. Otherwise, we can switch the edges between $M_1$ and $M_2$ for the component of $G'$ where $uu_3$ is in. Then we delete $u_2w$ from $M_2$ and add $uu_2$ to $M_2$. This is a contradiction since the number of edges in $H$ is dropped by one. 

Since $v, w, u$ are all degree one vertices in $G'$, $v$ cannot be in the same component with $w$ and $u$. Let the component $v$ is in be called $G_2'$. We can guarantee that $vw_1 \in M_1$ and $w_1w_3 \in M_2$ (possibly by switching edges in $G_2'$). Then we delete $vw_1$ from $M_1$ and add $ww_1$ to $M_1$. This is a contradiction since the number of edges in $H$ is dropped by one.
\end{proof}

\begin{lemma}\label{three2-paths-3}
The two leaves of a $P_3$ in $\hat{G}$  cannot be adjacent to the middle vertices of two other $P_3$s in $\hat{G}$  at the same time.
\end{lemma}

\begin{proof}
Let $u_1uu_2$, $v_1vv_2$, and $w_1ww_2$ be three $P_3$s in $\hat{G}$ such that $u_1v, u_2w \in M_1 \cup M_2$. Let $N(u) = \{u_1, u_2, u_3\}$, $N(u_3) = \{u, u_4, u_5\}$, $N(u_1) = \{u,v,v_3\}$, $N(v) = \{u_1, v_1, v_2\}$, $N(v_3) = \{u_1, v_4, v_5\}$, $N(u_2) = \{u,w,w_3\}$, $N(w) = \{u_2, w_1, w_2\}$, and $N(w_3) = \{u_2, w_4, w_5\}$ (See Figure~\ref{C_10} right picture). By Lemma~\ref{three2-paths-2}, one of $u_3u_4$ and $u_3u_5$ must be in $\hat{G}$. We may assume $u_3u_4 \in \hat{G}$. Without loss of generality, we may assume $uu_3 \in M_1$ and $u_3u_5 \in M_2$. We must have $u_2w \in M_1$ and $u_2u_3 \in M_2$ since otherwise we can delete $u_2w$ from $M_2$ and add $uu_2$ to $M_2$. This is a contradiction since the number of matching edges does not change but the number of edges in $H$ is dropped by one. Furthermore, $uu_3$ must be in the same component with $u_2w$ in $G'$ since otherwise we can switch the edges of $u_2w$ and $u_2w_3$ so that $u_2w \in M_2$ and $u_2u_3 \in M_1$ (then $uu_3$ will not change), which again will reach a contradiction. Note that $u_2w, u_3u,$ and $u_1v$ cannot be in the same component in $G'$ since $\Delta(G') \le 2$ and $v,w,u$ are all leaves. Therefore, we can assume $u_1v \in M_2$ and $u_1v_3 \in M_1$. We delete $u_1v$ from $M_2$ and add $uu_1$ to $M_2$. This is a contradiction since the number of matching edges does not change but the number of edges in $H$ is dropped by one.
\end{proof}

\section{The discharging rule and the final proof}

We only have one rule for the discharging method. We need to show~\eqref{charges} is true to obtain a contradiction.

\begin{quote}
\textbf{(R0):} Let $B_1, B_2$ be two basic components. If there is an edge $e_1 = uv \in M_1 \cup M_2$ such that $u$ is a $1$-vertex in $B_1$ and $v$ is a $2$-vertex in $B_2$, then $B_1$ gives $1$ to $B_2$.
\end{quote}

Note that a $P_2$ in $\hat{G}$ has degree $4+t$ in $H$, where $t$ is the number of $P_3$s in $\hat{G}$ whose middle vertices are adjacent to the $P_2$. Therefore, after Rule (R0), the maximum net charge for $P_2$ in $\hat{G}$ is $4+t-4.5-t=-0.5$. For other basic components, we can do a similar calculation and have the following table:
\textbf{
\begin{center}
\begin{tabular}{|c|c|c|}
\hline
    Basic Component & Total charge after R0 & Net charge after R0 \\ \hline
    $P_2$ & $4-4.5$ & $-0.5$ \\ \hline
    $P_3$ (not paired) & $9-4.5*2$ & $0$ \\ \hline
    $P_3$ (paired) & $5+5-4.5*2$ & $1$ \\ \hline
    $K_{1,3}$ & $4+4+4-4.5*3$ & $-1.5$ \\ \hline 
\end{tabular}
\end{center}
}
$$\text{Table 1: Charges for basic components.}$$

\textbf{Case 1:} There are $K_{1,3}$s and no paired $P_3$s. If we have at least two $K_{1,3}$s, then the total net charge is $-3<-2.5$. This is a contradiction since all other components are providing non-positive net charges. Thus, we may assume there is exactly one $K_{1,3}$. By Lemma~\ref{three2-paths},~\ref{three2-paths-2}, and~\ref{three2-paths-3}, we know there is a $P_2$ for every $14$ vertices or there is a $K_{1,3}$ for every 24 vertices. The former case can provide a net charge of $-0.5$ and the latter case can provide $-1.5$. Since we have more than $70$ vertices, the total net charge is less than $-2.5$.

\textbf{Case 2:} There are no paired $P_3$s and no $K_{1,3}$s. Hence, there are only $P_3$s and $P_2$s. By Lemma~\ref{three2-paths},~\ref{three2-paths-2}, and~\ref{three2-paths-3}, we know there is a $P_2$ for every $14$ vertices. Since we have more than $70$ vertices, there are at least $6$ $P_2$s in the graph. Thus, the total net charge is at most $-0.5*6 = -3<-2.5$.

\textbf{Case 3:} There is a pair of $P_3$s in $G$ that are joined by an edge in $M_1 \cup M_2$ via their middle vertices. By Lemma~\ref{two2pathsmiddle}, there is exactly one such pair and the structure of the graph is fixed. By Table 1, the two $P_3$s, $x_1u_1w_1$ and $y_1v_1z_1$, each has net charge $1$ after R0. For each $P_2$s (e.g. $v_6z_7$), it has net charge of $-0.5$ after R0 according to Table 1. Since we have more than $70$ vertices, if there are at least $10$ $P_2$s then they can provide in total $-0.5*10 = -5$ net charges after R0. This is a contradiction since in total we have the charge of at most $-5 + 2 = -3 < -2.5$. Hence, we only need to show there are at least $10$ $P_2$s.

The four branches in the left picture of Figure~\ref{C_final} starting with $x_1,y_1,z_1,w_1$ must exist. We first show the eight $P_2$s, $u_2x_7, u_4x_2, y_2v_4, y_7v_2, z_7v_6, z_2v_8, u_8w_2, u_6w_7$, in the left picture of Figure~\ref{C_final} must be distinct. Since $G$ is a subcubic graph, it is impossible that $u_2x_7 = u_4x_2$. Therefore, the only case that we have less $P_2$s is when some of the vertices from different branches coincide, say $v_2 = v_6$. However, we can switch the role of $y_1v_2$ and $y_1v_1$ to reduce the number of paired $P_3$s, which contradicts Condition (d). For a similar reason, the eight non-paired $P_3$s, $y_3v_5y_4, y_5v_3y_6, z_5v_7z_6, z_3v_9z_4, w_3u_9w_4, w_5u_7w_6, x_5u_3x_6, x_3u_5x_4$, are distinct. Since the local structure of the graph is fixed, there are only $P_2$s and $P_3$s in $G-M_1-M_2$, and the graph must grow in each direction, we can find two more $P_2$s.


\end{document}